\documentclass[a4paper,11pt]{amsart}

\usepackage{a4wide}
\usepackage{tikz}

\newcommand{\B}{\mathcal{B}}
\newcommand{\T}{\mathcal{T}}

\newcommand{\E}{\mathbb{E}}

\newcommand{\nut}{\nu^{(T)}}
\newcommand{\nup}{\nu^{(P)}}

\newif\ifdetails
\detailstrue
\newcommand{\DETAIL}[1]%
{\ifdetails\par\fbox{\begin{minipage}{0.9\linewidth}\textit{Detail:}
      #1\end{minipage}}\par\fi}
\newcommand{\TODO}[1]%
{\ifdetails\par\fbox{\begin{minipage}{0.9\linewidth}\textbf{TODO:}
      #1\end{minipage}}\par\fi}

\newtheorem{lemma}{Lemma}

\newtheorem{theorem}[lemma]{Theorem}
\newtheorem{corollary}[lemma]{Corollary}
\theoremstyle{remark}

\newtheorem{remark}{Remark}

\newcommand{\s}{{\mathfrak S}}

\title{The shape of random tanglegrams}

\author{Matja\v{z} Konvalinka}
\address{Matja\v{z} Konvalinka \\ Department of Mathematics, University of Ljubljana, and Institute for Mathematics, Physics and Mechanics \\ Jadranska 19 \\ Ljubljana \\ Slovenia}
\email{matjaz.konvalinka@fmf.uni-lj.si}
\thanks{Partially supported by Research Program Z1-5434 of the Slovenian Research Agency.}

\author{Stephan Wagner}
\address{Stephan Wagner\\
        Department of Mathematical Sciences\\
        Stellenbosch University\\
        Private Bag X1\\
        Matieland 7602\\
        South Africa}
\email{swagner@sun.ac.za}
\thanks{This material is based upon work supported financially by the National Research Foundation of South Africa under grant number 96236.}

\begin{document}

\begin{abstract}
A tanglegram consists of two binary rooted trees with the same number of leaves and a perfect matching between the leaves of the trees.
We show that the two halves of a random tanglegram essentially look like two independently chosen random plane binary trees. This fact is used to derive a number of results on the shape of random tanglegrams, including theorems on the number of cherries and generally occurrences of subtrees, the root branches, the number of automorphisms, and the height. For each of these, we obtain limiting probabilities or distributions. Finally, we investigate the number of matched cherries, for which the limiting distribution is identified as well.
\end{abstract}

\maketitle

\section{Introduction}

Tanglegrams are, intuitively, graphs obtained by taking two binary rooted trees with the same number of leaves (which is the \emph{size} of a tanglegram) and matching each leaf from the tree on the left with a unique leaf from the tree on the right. Furthermore, we consider two tanglegrams to be the same if we can get one from the other by an isomorphism that fixes the roots. For example, the following figure shows all 13 tanglegrams of size $4$.

\begin{figure}[h!]
\begin{center}
\begin{tikzpicture}[scale = 0.4]
\newcommand{\treeb}[3]{\coordinate (v1) at (#1,#3); \coordinate (v2) at (#1+#2,#3+0.5);\coordinate (v3) at (#1+2*#2,#3+1);\coordinate (v4) at (#1+3*#2,#3+1.5);\coordinate (v5) at (#1+3*#2,#3+0.5);\coordinate (v6) at (#1+3*#2,#3-0.5);\coordinate (v7) at (#1+3*#2,#3-1.5);\draw[fill] (v1) circle (.5ex);\draw[fill] (v4) circle (.5ex);\draw[fill] (v5) circle (.5ex);\draw[fill] (v6) circle (.5ex);\draw[fill] (v7) circle (.5ex);\draw (v1) -- (v4);\draw (v3) -- (v5);\draw (v2) -- (v6);\draw (v1) -- (v7);}
\newcommand{\treec}[3]{\coordinate (v1) at (#1,#3); \coordinate (v2) at (#1+2*#2,#3-1);\coordinate (v3) at (#1+2*#2,#3+1);\coordinate (v4) at (#1+3*#2,#3+1.5);\coordinate (v5) at (#1+3*#2,#3+0.5);\coordinate (v6) at (#1+3*#2,#3-0.5);\coordinate (v7) at (#1+3*#2,#3-1.5);\draw[fill] (v1) circle (.5ex);\draw[fill] (v4) circle (.5ex);\draw[fill] (v5) circle (.5ex);\draw[fill] (v6) circle (.5ex);\draw[fill] (v7) circle (.5ex);\draw (v1) -- (v4);\draw (v3) -- (v5);\draw (v2) -- (v6);\draw (v1) -- (v7);}
\newcommand{\tangleb}[7]{\treeb{#1}{#2}{#7} \treeb{#1+8*#2}{-#2}{#7} \draw[dashed] (#1 + 3*#2,#7+1.5) -- (#1 + 5*#2,#7+2.5-#3); \draw[dashed] (#1 + 3*#2,#7+0.5) -- (#1 + 5*#2,#7+2.5-#4); \draw[dashed] (#1 + 3*#2,#7-0.5) -- (#1 + 5*#2,#7+2.5-#5); \draw[dashed] (#1 + 3*#2,#7-1.5) -- (#1 + 5*#2,#7+2.5-#6);}
\newcommand{\tanglec}[7]{\treeb{#1}{#2}{#7} \treec{#1+8*#2}{-#2}{#7} \draw[dashed] (#1 + 3*#2,#7+1.5) -- (#1 + 5*#2,#7+2.5-#3); \draw[dashed] (#1 + 3*#2,#7+0.5) -- (#1 + 5*#2,#7+2.5-#4); \draw[dashed] (#1 + 3*#2,#7-0.5) -- (#1 + 5*#2,#7+2.5-#5); \draw[dashed] (#1 + 3*#2,#7-1.5) -- (#1 + 5*#2,#7+2.5-#6);}
\newcommand{\tangled}[7]{\treec{#1}{#2}{#7} \treeb{#1+8*#2}{-#2}{#7} \draw[dashed] (#1 + 3*#2,#7+1.5) -- (#1 + 5*#2,#7+2.5-#3); \draw[dashed] (#1 + 3*#2,#7+0.5) -- (#1 + 5*#2,#7+2.5-#4); \draw[dashed] (#1 + 3*#2,#7-0.5) -- (#1 + 5*#2,#7+2.5-#5); \draw[dashed] (#1 + 3*#2,#7-1.5) -- (#1 + 5*#2,#7+2.5-#6);}
\newcommand{\tanglee}[7]{\treec{#1}{#2}{#7} \treec{#1+8*#2}{-#2}{#7} \draw[dashed] (#1 + 3*#2,#7+1.5) -- (#1 + 5*#2,#7+2.5-#3); \draw[dashed] (#1 + 3*#2,#7+0.5) -- (#1 + 5*#2,#7+2.5-#4); \draw[dashed] (#1 + 3*#2,#7-0.5) -- (#1 + 5*#2,#7+2.5-#5); \draw[dashed] (#1 + 3*#2,#7-1.5) -- (#1 + 5*#2,#7+2.5-#6);}
\tangleb {0} {0.7} 1 2 3 4 {0}
\tangleb {7} {0.7} 1 2 4 3 {0}
\tangleb {14} {0.7} 1 3 2 4 {0}
\tangleb {21} {0.7} 1 3 4 2 {0}
\tangleb {28} {0.7} 1 4 2 3 {0}
\tangleb {0} {0.7} 1 4 3 2 {-5}
\tangleb {7} {0.7} 3 4 1 2 {-5}
\tanglec {14} {0.7} 1 2 3 4 {-5}
\tanglec {21} {0.7} 1 3 2 4 {-5}
\tangled {28} {0.7} 1 2 3 4 {-5}
\tangled {7} {0.7} 1 3 2 4 {-10}
\tanglee {14} {0.7} 1 2 3 4 {-10}
\tanglee {21} {0.7} 1 3 2 4 {-10}
\end{tikzpicture}
\end{center}
\caption{The $13$ tanglegrams of size $4$.}\label{fig:t4}
\end{figure}
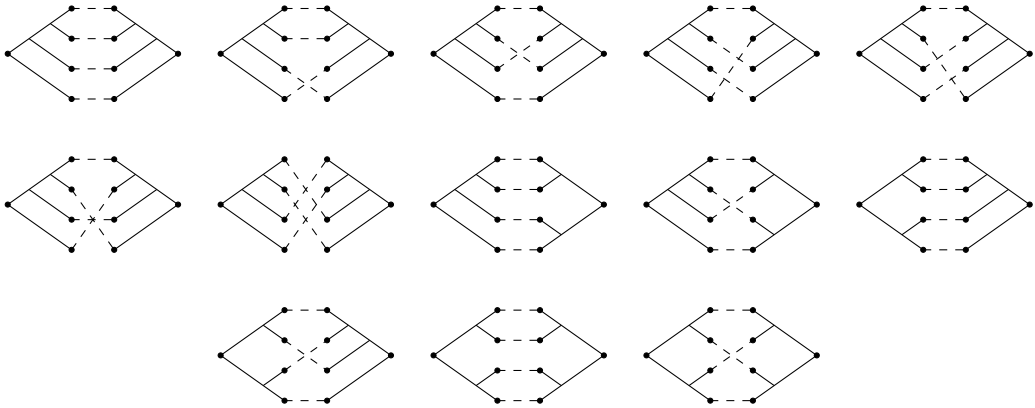

Let us make this definition more precise. A \emph{plane binary tree} has one distinguished vertex assumed to be a common ancestor of all other vertices, and each vertex either has two children (left and right) or no children.  A vertex with no children is a \emph{leaf}, and a vertex with two children is an \emph{internal vertex}. It is well known that the number of plane binary trees with $n$ leaves is the Catalan number $\frac{1}{n} \binom{2n-2}{n-1}$, henceforth denoted by $C_n$.  The sequence starts with
$$1, 1, 2, 5, 14, 42, 132, 429, 1430, 4862, 16796, 58786, 208012,$$
see \cite[A000108]{oeis} and \cite[\S 6]{ec2} for more information.

Two plane binary trees with labeled leaves are said to be \emph{equivalent} if there is an isomorphism from one to the other as graphs mapping the root of
one to the root of the other. Let $\B_n$ be the set of inequivalent plane binary trees with $n\geq 1$ leaves. In the following, we will refer to the elements of $\B_n$ merely as \emph{binary trees} for simplicity. The sets $\B_n$ are enumerated by the Wedderburn-Etherington numbers, whose sequence starts
$$1, 1, 1, 2, 3, 6, 11, 23, 46, 98, 207, 451, 983,$$
see \cite[A001190]{oeis} for more information.

For each plane binary tree $T$, denote by $A(T)$ its \emph{automorphism group}, which can be interpreted as a subgroup of the permutation group of the set of leaves, i.e.~ as a subgroup of $\s_n$. Given a permutation $v \in \s_n$ along with two trees $T,S \in \B_n$, each with leaves labeled $1,\ldots, n$, we construct an \emph{ordered binary rooted tanglegram} (or \emph{tanglegram} for short) $(T,v,S)$ of size $n$ with $T$ as the left tree, $S$ as the right tree, by identifying leaf $i$ in $T$ with leaf $v(i)$ in $S$.  Furthermore, $(T,v,S)$ and $(T',v',S')$ are considered to represent the same tanglegram provided $T = T'$,\ $S = S'$ as trees and $v'=uvw$, where $u\in A(T)$ and $w \in A(S)$. In other words, a tanglegram is a double coset of the symmetric group $\s_n$ with respect to the double action of $A(T)$ on the left and $A(S)$ on the right, where $T,S \in \B_n$.
 
Let $\T_n$ be the set of all tanglegrams of size $n$, and let $t_n$ be the number of elements in the set $\T_n$. The sequence starts
$$1,1, 1, 2, 13, 114, 1509, 25595, 535753, 13305590, 382728552, 12515198465, 458621603279,$$
see \cite[A258620]{oeis} for more terms. Note that in Figure~\ref{fig:t4}, the dashed lines are not technically part of the graph, but this visualization allows us to give a planar drawing of the two trees.

Tanglegrams naturally arise in biology, in particular in the study of cospeciation and coevolution. For example, the tree on the left may represent the phylogeny of a host,
such as gopher, while the tree on the right may represent a parasite, such as louse \cite{Hafner1988-da}, \cite[page 71]{Page.2002}. For more information on tanglegrams in biology, see \cite{arniePaper}.

In computer science, the Tanglegram Layout Problem (TL, see~\cite{bukin.etal.2008}) is to find a
drawing of a tanglegram in the plane with the left and right trees
both given as planar embeddings with the smallest number of crossings
among (straight) edges matching the leaves of the left tree and the
right tree. The authors of~\cite{bukin.etal.2008} point out that tanglegrams occur in the analysis of software projects and clustering problems.

It was recently shown in \cite{billey15} that the number of tanglegrams of size $n$ is given by
\begin{equation}\label{eq:explicit}
t_n = \sum_{\lambda} z_{\lambda} \Big( \sum_{T \in \B_n} \frac{|A(T)_{\lambda}|}{|A(T)|} \Big)^2,
\end{equation}
where the sum is over all \emph{binary partitions} $\lambda$ of $n$, i.e.~ partitions whose parts are all powers of $2$, $z(\lambda) = \prod_j (2^j)^{m_j} m_j!$ ($m_j$ being the number of occurrences of $2^j$ in $\lambda$), and $A(T)_{\lambda}$ is the set of automorphisms of $T$ whose conjugacy class is $\lambda$.

Moreover, it was found that 

$$
\sum_{T \in \B_n} \frac{|A(T)_{\lambda}|}{|A(T)|} = \frac{\prod_{i \geq 2} (2(\lambda_i + \cdots + \lambda_{\ell(\lambda)}) - 1)}{z_{\lambda}}.
$$
Here $\lambda_1,\lambda_2,\ldots,\lambda_{\ell(\lambda)}$ are the parts of $\lambda$ written in weakly decreasing order. For example, the binary partitions of $n=4$ are $4$, $22$, $211$ and $1111$, so
 $$t_4 = \frac{1}{4} + \frac{3^2}{8} + \frac{3^2 \cdot 1^2}{4} + \frac{5^2 \cdot 3^2 \cdot 1^2}{24} = 13.$$
 The formula was used together with~\eqref{eq:explicit} to obtain an asymptotic formula for $t_n$:
\begin{equation}\label{eq:asymp}
\frac{t_n}{n!} = e^{1/8} \Bigg( \frac{\frac{1}{n} \binom{2n-2}{n-1}}{2^{n-1}} \Bigg)^2 \big( 1 + O(n^{-1})\big) = \frac{e^{1/8} 4^{n-1}}{\pi n^3} \big( 1 + O(n^{-1})\big).
\end{equation}
See \cite[Corollary 8]{billey15} for more details. 

Once the enumeration problem is solved, it is very natural to consider random tanglegrams and study their shape. An algorithm that generates tanglegrams uniformly at random was described in \cite{billey15}, and a number of questions in this regard were put forward. The aim of this paper is to answer these questions. In fact, we will obtain them as corollaries of a rather precise structure theorem stating that the two halves of a random tanglegram look essentially like two independently chosen random plane binary trees.

In order to make the similarity between tanglegrams and pairs of plane binary trees precise, let us first recall the concept of the total variation distance of probability measures: for two such measures $\pi_1,\pi_2$ defined on the same $\sigma$-algebra $\mathcal{F}$, one defines
$$d(\pi_1,\pi_2) = \sup_{S \in \mathcal{F}} |\pi_1(S) - \pi_2(S)|.$$
The two probability measures we are comparing are both defined on the set $\B_n^2$ of pairs of binary trees. The first measure $\nut_n$ is the measure induced by the uniform measure on random tanglegrams (the two components simply being the two halves of the tanglegram), the other one, denoted by $\nup_n$, is the measure obtained by choosing two plane binary trees uniformly and independently at random. Our main theorem reads as follows:

\begin{theorem}\label{thm:main}
The total variation distance $d(\nut_n,\nup_n)$ goes to $0$ as $n \to \infty$, specifically $d(\nut_n,\nup_n) = O(n^{-1/2})$. Moreover, there exist positive constants $M_1$ and $M_2$ such that we have
\begin{equation}\label{eq:upper_bound}
\nut_n(S) \leq M_1\nup_n(S) + O(n^{-1}) \quad \text{and}\quad \nup_n(S) \leq M_2\nut_n(S) + O(n^{-1})
\end{equation}
for every subset $S$ of $\B_n^2$.
\end{theorem}

The first statement means that the two probability measures are globally very close, while the second one provides an estimate on how much the probability of very unlikely events can differ. Theorem~\ref{thm:main} allows us to carry over many structural properties from random plane binary trees to random tanglegrams. For example, the number of cherries (pairs of leaves sharing a common parent) in one half of a random tanglegram of size $n$ is sharply concentrated around $n/4$  and satisfies a central limit theorem; see Section~\ref{sec:conseq} for several further corollaries that follow from Theorem~\ref{thm:main}. First, however, we prove our main theorem in the following section. As it turns out, cherries play a key role in our analysis.

\section{Proof of the main result}

The first step is to rewrite the total variation distance, which is given by
$$d(\nut_n,\nup_n) = \sup_{S \subseteq \B_n^2} |\nut_n(S) - \nup_n(S)| =  \frac12 \sum_{(B_1,B_2) \in \B_n^2} \big|\nut_n(B_1,B_2) - \nup_n(B_1,B_2)\big|.$$
Let us remark here that for convenience we simply write $\nut_n(B_1,B_2)$ and $\nup_n(B_1,B_2)$ for the respective probabilities of the event that the pair $(B_1,B_2)$ is generated. The probability that a certain pair $(B_1,B_2) \in \B_n^2$ of binary trees are the left and right half of a random tanglegram is given by
$$\nut_n(B_1,B_2) = \frac{1}{t_n} \sum_{\lambda} z_{\lambda} \frac{|A(B_1)_{\lambda}|}{|A(B_1)|} \frac{|A(B_2)_{\lambda}|}{|A(B_2)|},$$
while the probability that a pair of randomly chosen plane binary trees is isomorphic to $(B_1,B_2)$ is given by
$$\nup_n(B_1,B_2) =  \Big( \frac{2^{n-1}}{C_n} \Big)^2 \cdot \frac{1}{|A(B_1)||A(B_2)|}.$$
Next we make use of the observation that only partitions $\lambda$ of the form $\lambda = 2^s1^{n-2s}$ matter. We denote the class of all such partitions by $R$. As it turns out, partitions that do not belong to $R$ are asymptotically irrelevant: to be precise, as it was shown in \cite{billey15},
$$\sum_{B_1,B_2} \frac{1}{t_n} \sum_{\lambda \not\in R} z_{\lambda} \frac{|A(B_1)_{\lambda}|}{|A(B_1)|} \frac{|A(B_2)_{\lambda}|}{|A(B_2)|}  = O(n^{-1}).$$
Thus we can restrict ourselves to summations over elements of $R$ in the following:
$$d(\nut_n,\nup_n) = \frac12 \sum_{B_1,B_2} \left| \frac{1}{t_n} \sum_{\lambda \in R} z_{\lambda} \frac{|A(B_1)_{\lambda}|}{|A(B_1)|} \frac{|A(B_2)_{\lambda}|}{|A(B_2)|} - \nup_n(B_1,B_2)\right| + O(n^{-1}).$$
Using the asymptotic formula~\eqref{eq:asymp}, we obtain
$$\frac{1}{t_n |A(B_1)| |A(B_2)|} = \nup_n(B_1,B_2) \cdot \frac{1}{e^{1/8} n!} \cdot (1+O(n^{-1}))$$
 for any pair $(B_1,B_2)$, which gives us
\begin{equation}\label{eq:tot_dist}
d(\nut_n,\nup_n) = \frac12 \sum_{B_1,B_2} \nup_n(B_1,B_2) \left| \frac{1}{e^{1/8} n!} \sum_{\lambda \in R} z_{\lambda} |A(B_1)_{\lambda}||A(B_2)_{\lambda}| - 1 \right| + O(n^{-1}).
\end{equation}
It remains to show that for a randomly chosen pair $B_1,B_2$ of plane binary trees, the expression inside the absolute value bars is small. This will be achieved in a sequence of lemmas. The first provides information on the size of $|A(T)_{\lambda}|$ for $\lambda \in R$:

\begin{lemma}\label{lem:auto_bounds}
Let $\mu(s) = 2^s 1^{n-2s}$ be the partition of $n$ consisting of $s$ twos and $n-2s$ ones, and let $c(T)$ denote the number of cherries of a binary tree $T$. We have the inequalities
$$\binom{c(T)}{s} \leq |A(T)_{\mu(s)}| \leq \binom{c(T)+s-1}{s}.$$
\end{lemma}
\begin{proof}
Intuitively speaking, most automorphisms whose cycle type is of the form $\mu(s) = 2^s 1^{n-2s}$ are obtained by picking $s$ cherries and interchanging the leaves of each of these cherries and nothing else. In the following, we make this heuristic precise by providing some explicit inequalities. The inequality
$$|A(T)_{\mu(s)}| \geq \binom{c(T)}{s}$$
is immediate, since there is indeed one automorphism of the desired cycle type $\mu(s)$ for every possible choice of $s$ cherries. For the proof of the other inequality
\begin{equation}\label{eq:auto-upper}
|A(T)_{\mu(s)}| \leq \binom{c(T)+s-1}{s},
\end{equation}
we first define the following polynomial associated with a binary tree $T$:
$$P(T,u) = \sum_{s \geq 0} |A(T)_{\mu(s)}| u^s.$$
Suppose that the two branches of $B$ are $T_1$ and $T_2$ respectively. It is easy to see that
$$P(T,u) = P(T_1,u) P(T_2,u)$$
if $T_1$ and $T_2$ are distinct, and
$$P(T,u) = P(T_1,u)^2 + |A(T_1)| u^{|T_1|}$$
if $T_1$ and $T_2$ are identical (isomorphic). We want to show that the coefficient of $u^s$ in $P(T,u)$ is always less than or equal to $\binom{c(T)+s-1}{s}$, which is the corresponding coefficient in the power series expansion of $(1-u)^{-c(T)}$. Denoting coefficient-wise inequality of polynomials or power series by $\preceq$, we can express this as
$$P(T,u) \preceq (1-u)^{-c(T)},$$
which we now proceed to prove by induction on the size of $T$. For a tree that consists only of a single leaf or a single cherry, the inequality is obvious. We also remark on this occasion that the degree of $P(T,u)$ is at most $|T|/2$.

For the induction step, we first consider the easy case that the two branches are distinct. In this case, we have
$$P(T,u) = P(T_1,u) P(T_2,u) \preceq (1-u)^{-c(T_1)}(1-u)^{-c(T_2)} = (1-u)^{-c(T)},$$
and we are done. If the two branches are identical, then we have to be more careful: since
$$P(T,u) = P(T_1,u)^2 + |A(T_1)| u^{|T_1|},$$
the coefficient-wise inequality follows in the same way, except perhaps for the coefficient of $u^{|T_1|}$, where we need to verify it directly. Write $t = |T_1|$ for the size (number of leaves) of $T_1$. Since the degree of $P(T_1,u)$ is at most $t/2$, the contribution of $P(T_1,u)^2$ to the coefficient of $u^t$ is restricted to the squared coefficient of $u^{t/2}$ in $P(T_1,u)$ (if there is such a coefficient). Applying the induction hypothesis, we find that the coefficient of $u^t$ in $P(T,u)$ is at most
$$\binom{c(T_1) + t/2 -1}{t/2}^2 + |A(T_1)|,$$
where we interpret the binomial coefficient as $0$ if $t$ is odd. Next, we observe that $|A(T_1)| \leq 2^{2c(T_1) - 1}$, which can be obtained by another easy induction from the recursion
$$|A(T)| = |A(T_1)||A(T_2)|$$
if the branches $T_1$ and $T_2$ of $T$ are distinct, and
$$|A(T)| = 2|A(T_1)|^2$$
if they are identical. Hence we have
$$\binom{c(T_1) + t/2 -1}{t/2}^2 + |A(T_1)| \leq \binom{c(T_1) + t/2 -1}{t/2}^2 + 2^{2c(T_1) - 1},$$
and we have to show that this is less than or equal to $\binom{c(T) + t -1}{t}$. This can be achieved in different ways, one of them being combinatorial. First of all, note that $c(T_1) \leq t/2$. The binomial coefficient $\binom{c(T) + t - 1}{t} = \binom{2c(T_1) + t-1}{t}$ gives the number of ways to select $t$ elements from the set $1,2,\ldots,2c(T_1)$, repetitions allowed. The expression $\binom{c(T_1) + t/2 -1}{t/2}^2$ counts those choices for which the same number of elements are taken from the first and from the second half (zero if $t$ is odd). Now we associate to every subset of $\{1,2,\ldots,2c(T_1)-1\}$, of which there are exactly $2^{2c(T_1)-1}$, a possible selection where the numbers in the two halves are not the same: this is done by adding an appropriate number of copies (i.e.~as many as needed to obtain a multiset of $t$ elements) of the element $2c(T_1)$ to the subset unless this creates a ``balanced'' selection, in which case we increase the number of copies of the least element (which must be in the first half) appropriately instead. This clearly creates an injection that proves the desired inequality and hence completes the induction proof of~\eqref{eq:auto-upper}.
\end{proof}

Now we apply our estimates to deal with the expressions that occur in~\eqref{eq:tot_dist} in the following lemma:

\begin{lemma}\label{lem:cherry}\
\begin{enumerate}
\item\label{cherry1} There exists an absolute constant $K$ such that 
$$\frac{1}{n!} \sum_{\lambda \in R} z_{\lambda} |A(B_1)_{\lambda}||A(B_2)_{\lambda}| \leq K$$
for all possible pairs $(B_1,B_2)$ of binary trees with $n$ leaves.
\item\label{cherry2} If we further assume that $c(B_1),c(B_2) \geq \alpha n$ for some fixed constant $\alpha$, then
$$ \frac{1}{n!} \sum_{\lambda \in R} z_{\lambda} |A(B_1)_{\lambda}||A(B_2)_{\lambda}| = \exp \Big( \frac{2c(B_1)c(B_2)}{n^2} \Big)+ O(n^{-1}),$$
where the constant implied by the $O$-term only depends on $\alpha$.
\end{enumerate}
\end{lemma}
\begin{proof}
From the previous lemma, we know that
\begin{align*}
\frac{1}{n!} \sum_{\lambda \in R} z_{\lambda} |A(B_1)_{\lambda}||A(B_2)_{\lambda}| &\leq \frac{1}{n!} \sum_{0 \leq s \leq n/2} z_{\mu(s)} \binom{c(T_1)+s-1}{s}\binom{c(T_2)+s-1}{s} \\
&\leq \frac{1}{n!} \sum_{0 \leq s \leq n/2} z_{\mu(s)} \binom{n/2+s-1}{s}^2,
\end{align*}
and we would like to show that this is bounded by an absolute constant. By definition, $z_{\mu(s)} = 2^s s!(n-2s)!$, so it remains to bound
$$\frac{1}{n!} \sum_{0 \leq s \leq n/2} 2^s s! (n-2s)! \binom{n/2+s-1}{s}^2.$$
To this end, we apply the simple inequality
$$(n-2s)! = n! \prod_{j=0}^{2s-1} (n-j)^{-1} \leq n! (n-2s)^{-2s}$$
and split the sum into three parts: the first part is
\begin{align*}
\frac{1}{n!} \sum_{0 \leq s \leq \sqrt{n}} 2^s s! (n-2s)! \binom{n/2+s-1}{s}^2 
&\leq \sum_{0 \leq s \leq \sqrt{n}} 2^s s! \frac{(n-2s)!}{n!} \frac{(n/2+s-1)^{2s}}{s!^2} \\
&\leq \sum_{0 \leq s \leq \sqrt{n}} 2^s (n-2\sqrt{n})^{-2s} \frac{(n/2+\sqrt{n})^{2s}}{s!} \\
&= \sum_{0 \leq s \leq \sqrt{n}} \frac{2^{-s}}{s!} \Big(1 - \frac{2}{\sqrt{n}}\Big)^{-2s} \Big(1 + \frac{2}{\sqrt{n}}\Big)^{2s} \\
&\leq \sum_{s \geq 0} \frac{2^{-s}}{s!} \Big( \frac{\sqrt{n}+2}{\sqrt{n}-2} \Big)^{2s} \\
&= \exp \Big( \frac{(\sqrt{n}+2)^2}{2(\sqrt{n}-2)^2} \Big),
\end{align*}
which converges to $e^{1/2}$ as $n \to \infty$ and is therefore bounded. Likewise, since $n-2s > n/2+s-1$ if $s \leq n/6$,
\begin{align}
\frac{1}{n!} \sum_{\sqrt{n} < s \leq n/6} 2^s s! (n-2s)! \binom{n/2+s-1}{s}^2 
&\leq \sum_{\sqrt{n} < s \leq n/6} 2^s s! (n-2s)^{-2s} \frac{(n/2+s-1)^{2s}}{s!^2}  \nonumber \\
&\leq \sum_{\sqrt{n} < s \leq n/6} \frac{2^s}{s!} = O \Big( \frac{2^{\lceil\sqrt{n}\rceil}}{\lceil\sqrt{n}\,\rceil !} \Big), \label{eq:est1}
\end{align}
which goes to $0$ as $n \to \infty$ and is therefore also bounded. Finally,
\begin{align}
\frac{1}{n!} \sum_{n/6 < s \leq n/2} 2^s s! (n-2s)! \binom{n/2+s-1}{s}^2 
&\leq \sum_{n/6 < s \leq n/2} \frac{2^s}{\binom{n-s}{s}}  \frac{(n-s)!}{n!} (2^{n/2+s-1})^2 \nonumber \\
&\leq \sum_{n/6 < s \leq n/2} 2^{n/2} (n/2)^{-s} 2^{2n} \nonumber \\
&\leq (n/2) \cdot 2^{5n/2} (n/2)^{-n/6}, \label{eq:est2}
\end{align}
which also goes to $0$ as $n \to \infty$. This completes the proof of the first part of Lemma~\ref{lem:cherry}.

For the rest of the proof, we assume that $c(B_1),c(B_2) \geq \alpha n$. Now the inequalities
$$\binom{c(B_i)}{s} \leq |A(B_i)_{\mu(s)}| \leq \binom{c(B_i) + s-1}{s}$$
imply
$$|A(B_i)_{\mu(s)}| = \frac{c(B_i)^s}{s!} \Big( 1 + O(s^2/n) \Big)$$
whenever $s \leq \sqrt{n}$. We already know from the estimates~\eqref{eq:est1} and~\eqref{eq:est2} that the partitions $\mu(s)$ with $s > \sqrt{n}$ only contribute a small error term (even much less than $O(n^{-1})$) to the sum
$$ \frac{1}{n!} \sum_{\lambda \in R} z_{\lambda} |A(B_1)_{\lambda}||A(B_2)_{\lambda}|,$$
so we can focus on the values of $s$ with $s \leq \sqrt{n}$. For those, we also have
$$(n-2s)! = n! \cdot n^{-2s}  \Big( 1 + O(s^2/n) \Big).$$
Putting everything together yields
\begin{align*}
\frac{1}{n!} \sum_{\lambda \in R} z_{\lambda} |A(B_1)_{\lambda}||A(B_2)_{\lambda}| &= \frac{1}{n!} \sum_{0 \leq s \leq \sqrt{n}} 2^s s! (n-2s)! |A(B_1)_{\mu(s)}||A(B_2)_{\mu(s)}| + O(n^{-1}) \\
&= \sum_{0 \leq s \leq \sqrt{n}} 2^s s!  n^{-2s}  \frac{c(B_1)^s}{s!} \frac{c(B_2)^s}{s!} \Big( 1 + O(s^2/n) \Big) + O(n^{-1}) \\
&= \sum_{s \geq 0} \frac{1}{s!} \Big(\frac{2c(B_1)c(B_2)}{n^2} \Big)^s - \sum_{s > \sqrt{n}} \frac{1}{s!} \Big(\frac{2c(B_1)c(B_2)}{n^2} \Big)^s \\
&\qquad +O \bigg( n^{-1} \sum_{s \geq 0} \frac{s^2}{s!} \Big(\frac{2c(B_1)c(B_2)}{n^2} \Big)^s \bigg) + O(n^{-1}) \\
&= \exp \Big( \frac{2c(B_1)c(B_2)}{n^2} \Big) + O(n^{-1}).
\end{align*}
Note here that $2c(B_1)c(B_2)/n^2$ is always bounded above by $\frac12$, so the infinite sums are easily estimated. This completes the proof of Lemma~\ref{lem:cherry}.
\end{proof}

Now we can return to the proof of Theorem~\ref{thm:main}. It is well known that the number of cherries in a random plane binary tree with $n$ leaves asymptotically follows a normal distribution, with mean $n(n-1)/(4n-6) \sim n/4$ and variance $O(n)$ (see \cite[Examples III.14 and IX.25]{flajolet09}). Thus if $B$ is a random plane binary tree with $n$ leaves, then by Chebyshev's inequality $|c(B) -n/4| > k$ holds with probability at most $O(n/k^2)$. Taking $k = n/8$, we find that $c(B) < n/8$ only occurs with probability $O(n^{-1})$ when a plane binary tree $B$ is selected uniformly at random (in fact, it is possible to obtain better estimates, but this is enough for our purposes). If either $c(B_1) < n/8$ or $c(B_2) < n/8$, we estimate the sum
$$\frac{1}{n!} \sum_{\lambda \in R} z_{\lambda} |A(B_1)_{\lambda}||A(B_2)_{\lambda}|$$
by means of the first part of Lemma~\ref{lem:cherry}, showing that this case only contributes $O(n^{-1})$ to the expression in~\eqref{eq:tot_dist}. Otherwise, we can use the second part of Lemma~\ref{lem:cherry} to obtain
\begin{align*}
d(\nut_n,\nup_n) &= \frac12 \sum_{B_1,B_2} \nup_n(B_1,B_2) \left| \frac{1}{e^{1/8} n!} \sum_{\lambda \in R} z_{\lambda} |A(B_1)_{\lambda}||A(B_2)_{\lambda}| - 1 \right| + O(n^{-1}) \\
&= \frac12 \sum_{B_1,B_2} \nup_n(B_1,B_2) \left| \exp \Big( \frac{2c(B_1)c(B_2)}{n^2} - \frac18 \Big) - 1 \right| + O(n^{-1}).
\end{align*}
The sum can be interpreted as the expected value of
$$\left| \exp \Big( \frac{2c(B_1)c(B_2)}{n^2} - \frac18 \Big) - 1 \right|$$
with respect to the measure $\nup_n$. Since
$$\left| \exp \Big( 2x_1 x_2  - \frac18 \Big) - 1 \right| = O \Big( \Big|x_1 - \frac14 \Big| + \Big| x_2 - \frac14 \Big| \Big)$$
for bounded $x_1,x_2$ (and $x_1 = c(B_1)/n$ and $x_2 = c(B_2)/n$ are indeed bounded above and below), we can estimate the total variation distance by
$$d(\nut_n,\nup_n) = O \bigg( \E^{(P)}_n \Big( \Big| \frac{c(B_1)}{n} - \frac14 \Big| + \Big| \frac{c(B_2)}{n} - \frac14 \Big| \Big) + n^{-1} \bigg),$$
where $\E_n^{(P)}$ denotes the expected value with respect to $\nup_n$. Now we can use independence of $B_1$ and $B_2$. Letting $\E_n^{(B)}$ denote the expected value when a plane binary tree $B$ is chosen uniformly at random, we have
$$\E_n^{(P)} \Big( \Big| \frac{c(B_1)}{n} - \frac14 \Big| + \Big| \frac{c(B_2)}{n} - \frac14 \Big| \Big) = 2 \E_n^{(B)} \Big( \Big| \frac{c(B)}{n} - \frac14 \Big| \Big),$$
and Jensen's inequality gives us
$$\E_n^{(B)} \Big( \Big| \frac{c(B)}{n} - \frac14 \Big| \Big) \leq \bigg( \E_n^{(B)} \Big( \Big( \frac{c(B)}{n} - \frac14 \Big)^2 \Big)\bigg)^{1/2} = O(n^{-1/2})$$
by the aforementioned fact that the variance of the number of cherries is only of linear order. Putting everything together, we obtain
\begin{equation}\label{eq:final_est}
d(\nut_n,\nup_n) = O(n^{-1/2}),
\end{equation}
which is exactly the first part of Theorem~\ref{thm:main}.

\begin{remark}
Since the typical fluctuations of $c(B_1)$ and $c(B_2)$ are of order $\sqrt{n}$, the proof also shows that the order of magnitude of our estimate is best possible, i.e.~the exponent $\frac12$ in~\eqref{eq:final_est} cannot be increased.
\end{remark}

Now we attend to the second statement of Theorem~\ref{thm:main}, namely~\eqref{eq:upper_bound}, which is somewhat easier to prove. We first observe that
$$\nut_n(S) = \sum_{(B_1,B_2) \in S} \nup_n(B_1,B_2) \frac{1}{e^{1/8} n!} \sum_{\lambda \in R} z_{\lambda} |A(B_1)_{\lambda}||A(B_2)_{\lambda}| + O(n^{-1})$$
by the same argument that gave us~\eqref{eq:tot_dist}. Part~(\ref{cherry1}) of Lemma~\ref{lem:cherry} guarantees that
$$\frac{1}{e^{1/8} n!} \sum_{\lambda \in R} z_{\lambda} |A(B_1)_{\lambda}||A(B_2)_{\lambda}| $$
is bounded above by an absolute constant, while the trivial estimate
$$\frac{1}{e^{1/8} n!} \sum_{\lambda \in R} z_{\lambda} |A(B_1)_{\lambda}||A(B_2)_{\lambda}| \geq \frac{1}{e^{1/8}}$$
is obtained by only taking the partition $\lambda = \mu(0) = 1^n$ consisting solely of ones into account (we have $z_{\lambda} = n!$ and $|A(B_1)_{\lambda}| = |A(B_2)_{\lambda}| = 1$ for this particular choice of $\lambda$). The inequalities~\eqref{eq:upper_bound} follow immediately.

\section{Consequences of the main theorem}\label{sec:conseq}

Our main theorem tells us that the probability distributions $\nut_n$ and $\nup_n$ are almost the same. It follows immediately that the behaviour of various shape parameters can be carried over from plane binary trees. As a first instance, we obtain Conjecture 2 of \cite{billey15} (as well as the more specific Conjecture $1$, which only considers cherries), which deals with the number of copies of a fixed rooted binary tree $B$ occurring as a fringe subtree (i.e.~a subtree consisting of a vertex and all its successors) in one half of a random tanglegram.

\begin{corollary}\label{cor:occurrences}
The average number of cherries in the left (or right) tree of a random tanglegram of size $n$ is asymptotically equal to $n/4$; generally, the average number of occurrences of a fixed binary tree $B$ is asymptotically equal to $\mu_Bn$, where the constant $\mu_B$ is given by $2^{1-|B|}/|A(B)|$. Moreover, the number of occurrences is asymptotically normally distributed: if $X_{B,n}$ denotes the number of occurrences of $B$ in the left half of a random tanglegram of size $n$, then we have
$$\lim_{n \to \infty} \nut_n \big( X_{B,n} \leq \mu_B n + x \sigma_B \sqrt{n} \big)  = \frac{1}{\sqrt{2\pi}} \int_{-\infty}^x e^{-t^2/2} \,dt$$
for every real $x$, where the constant $\sigma_B$ is defined by $\sigma^2_B = 2^{1-|B|}/|A(B)| + 4^{1-|B|}(1-2|B|)/|A(B)|^2$. In particular, for cherries $C$ we have $\sigma_C = \frac14$.
\end{corollary}

\begin{proof}
This is a consequence of the analogous statement for plane binary trees, which is obtained by standard means, cf.~\cite[Section 3.3]{drmota09}. There are $2^{|B|-1}/|A(B)|$ plane binary trees isomorphic to a binary tree $B$. Therefore, the bivariate generating function $Y(x,u)$ for plane binary trees, where the exponent of $x$ marks the number of leaves and $u$ the number of occurrences of $B$, satisfies the functional equation
$$Y(x,u) = x + Y(x,u)^2 + (u-1) \cdot \frac{2^{|B|-1}}{|A(B)|} x^{|B|}.$$
This can be explained from the observation that the number of occurrences of $B$ is either the sum of the occurrences in the two branches, or $1$ if the tree itself is isomorphic to $B$.
The central limit theorem and the asymptotic formula for the mean (for plane binary trees) now follow immediately from \cite[Theorem 2.23]{drmota09}.


Once the statement has been established for plane binary trees, it is easily carried over to tanglegrams: the probability of the event $X_{B,n} \leq \mu_B n + x \sigma_B \sqrt{n}$ can only change by $O(n^{-1/2})$ in view of Theorem~\ref{thm:main}, so the central limit theorem follows trivially. As for the mean value, we note that the number of occurrences is clearly $O(n)$, so the change in the mean from plane binary trees to tanglegrams is at most $O\big(n \cdot d(\nut_n,\nup_n)\big) = O(\sqrt{n})$, which does not affect the main term. 
\end{proof}

A similar corollary provides information about the root branches:

\begin{corollary}
The limiting probability that one of the root branches of the left (or right) tree of a random tanglegram consists of a single leaf is $\frac12$. Generally, the limiting probability that a fixed binary tree $B$ occurs as one of the two root branches of the left tree of a random tanglegram is $2^{-|B|}/|A(B)|$.
\end{corollary}

\begin{proof}
This can be shown by a simple direct calculation: again, we note first that there are $2^{|B|-1}/|A(B)|$ plane binary trees isomorphic to a binary tree $B$. Thus for large enough $n$ (greater than $2|B|$), the number of plane binary trees for which one of the two branches is isomorphic to $|B|$ is $2 \cdot 2^{|B|-1}/|A(B)| \cdot C_{n-|B|}$. Dividing by the total number of plane binary trees and taking the limit gives the desired result for plane binary trees, and it follows automatically for tanglegrams:
$$\lim_{n \to \infty} \frac{2 \cdot 2^{|B|-1}/|A(B)| \cdot C_{n-|B|}}{C_n} = \frac{2^{-|B|}}{|A(B)|}.$$
\end{proof}

Our next corollary is concerned with the height, i.e.~the length of the longest path from the root to a leaf. Again, the result for plane binary trees carries over directly.

\begin{corollary}
The average height of the left (or right) tree of a random tanglegram is asymptotically equal to $2\sqrt{\pi n}$, and the height asymptotically follows the theta distribution: if $H_n$ denotes the height of the left half of a random tanglegram of size $n$, then we have
$$\lim_{n \to \infty} \nut_n \big( H_n \geq x\sqrt{n} \big) = \Theta(x) = \sum_{j \geq 1} e^{-j^2 x^2}(4j^2x^2-2)$$
for every positive real number $x$.
\end{corollary}

\begin{proof}
Again, the statement on the limiting distribution follows from the analogous statement for plane binary trees (\cite[Proposition VII.16]{flajolet09}). For the mean value, we need to be slightly more careful than in the proof of Corollary~\ref{cor:occurrences}, since trees whose height is of linear order might a priori create an error term of order $\sqrt{n}$. However, \cite[Theorem 1.3]{flajolet93} guarantees that the probability for the height of a random plane binary tree to be greater than $h$ is $O(n^{3/2} e^{-h^2/(4n)})$. We apply this to $h = n^{2/3}$ (for example) and combine it with~\eqref{eq:upper_bound}. It follows that the height of one half of a random tanglegram is greater than $h$ with probability $O(n^{-1})$. So those trees only contribute at most $O(1)$ to the average height. Now we can apply the first part of Theorem~\ref{thm:main} to see that the average height only changes by $O(n^{2/3} \cdot n^{-1/2}) = O(n^{1/6})$ from plane binary trees to tanglegrams. Since the average height of a plane binary tree is asymptotically equal to $2\sqrt{\pi n}$ (\cite[Theorem B]{flajolet82}, see also \cite[Corollary 1.1]{flajolet93}), the corollary follows.
\end{proof}

Finally, we consider the automorphism group, more specifically the number of generators (which is the number of vertices for which the two branches are identical) of one half $B_1$ of a random tanglegram. We denote this parameter by $g(B_1)$ and note that $|A(B_1)| = 2^{g(B_1)}$, so this also provides information on the size of the automorphism group. The parameter $g$ was studied for plane binary trees by B\'ona and Flajolet in \cite{bona09}. We obtain the following result, which settles Conjecture 3 of \cite{billey15}:

\begin{corollary}
The expected number of generators of the left (or right) half of a random tanglegram of size $n$ is asymptotically equal to $\gamma n$, where the constant $\gamma$, whose numerical value is $0.2710416936\ldots$, is the value of the function $f(x)$ defined by $f(x) = x + \frac12 f(x)^2 + (x-\frac12) f(x^2)$ at $x = \frac14$. Moreover, the number of generators is asymptotically normally distributed.
\end{corollary}

\begin{proof}
In the same way as Corollary~\ref{cor:occurrences}, this follows from the analogous statement for plane binary trees, see \cite[Theorem 2, (ii)]{bona09}.
\end{proof}

\section{The number of matched cherries}

As it was pointed out in~\cite{billey15}, cherries play a major role in the literature on tanglegrams, and it was asked what the expected number of matched cherries (two cherries whose leaves are matched to each other) in a random tanglegram would be. We prove that this expected number is generally very small: in the limit, it follows a Poisson distribution whose mean and variance are $\frac14$. This result and its proof also help to better understand the asymptotic formula~\eqref{eq:asymp}. An important observation in this regard is the fact that each matched cherry induces a nontrivial automorphism, so $k$ matched cherries yield at least $2^k$ different automorphisms.

\begin{theorem}\label{thm:matched_cherries}
For every positive integer $k$, the probability that there are exactly $k$ matched cherries in a random tanglegram of size $n$ converges to $e^{-1/4} 4^{-k}/k!$, i.e.~the number of matched cherries has a limiting Poisson distribution. Moreover, the probability that there are exactly $k$ matched cherries and no more than the $2^k$ automorphisms induced by these converges to the same limit. Hence with high probability all automorphisms of a random tanglegram are induced by matched cherries.
\end{theorem}

\begin{proof}
Consider any tanglegram $T$ of size $n$; the group generated by rotations of all $2(n-1)$ internal vertices (by a rotation, we simply mean the action of interchanging the two branches) acts on the set of all possible representations of this tanglegram as a pair of two plane trees with a perfect matching between the leaves. The number of orbits, which is the number of distinct representations of this kind, is $2^{2(n-1)}/|A(T)|$ by the orbit-stabiliser theorem. Conversely, we can construct a tanglegram $T$ from a pair of two plane binary trees $B_1,B_2$ and a perfect matching $\sigma$, to which we assign a weight $|A(T)|/2^{2(n-1)}$. In this way, each distinct tanglegram is counted with a total weight of $1$ if we sum over all choices of $B_1$, $B_2$, and $\sigma$.

Now we use this way of counting tanglegrams as a means to estimate the probability of the event that there are exactly $k$ tangled cherries. Pick two plane binary trees $B_1$ and $B_2$ with $n$ leaves, and suppose that they have $c_1$ and $c_2$ cherries, respectively. Now we count the number of perfect matchings between the leaves of the two trees that generate exactly $k$ matched cherries, where $k$ is a fixed nonnegative integer. By a straightforward application of the inclusion-exclusion principle, this number is
\begin{multline*}
\binom{c_1}{k} \binom{c_2}{k} \cdot k! 2^k \cdot \sum_{\ell \geq 0} (-1)^{\ell} \binom{c_1-k}{\ell} \binom{c_2-k}{\ell} \cdot \ell! 2^{\ell} (n-2k-2\ell)! \\
= \frac{n! 2^k}{k!} \sum_{\ell \geq 0} \frac{(-1)^{\ell} 2^{\ell}}{\ell!}  \frac{\prod_{j=0}^{k+\ell-1} (c_1-j)(c_2-j)}{\prod_{j=0}^{2k+2\ell-1} (n-j)}.
\end{multline*}
The sum is estimated in a similar way as in the proof of part~(\ref{cherry2}) of Lemma~\ref{lem:cherry}. If we assume that $c_1,c_2 \geq n/8$ (which we already know to be the case for most choices of $B_1$ and $B_2$), then for $\ell \leq \sqrt{n}$, we have
$$\frac{\prod_{j=0}^{k+\ell-1} (c_1-j)(c_2-j)}{\prod_{j=0}^{2k+2\ell-1} (n-j)} = \Big( \frac{c_1c_2}{n^2} \Big)^{k+\ell} \Big(1 + O(\ell^2/n) \Big).$$
Moreover, since $c_1,c_2 \leq n/2$, it also follows easily that this fraction is bounded above by $1$. Thus
\begin{align*}
\sum_{\ell \geq 0} \frac{(-1)^{\ell} 2^{\ell}}{\ell!}  &\frac{\prod_{j=0}^{k+\ell-1} (c_1-j)(c_2-j)}{\prod_{j=0}^{2k+2\ell-1} (n-j)}
= \! \sum_{0 \leq \ell \leq \sqrt{n}} \! \frac{(-1)^{\ell} 2^{\ell}}{\ell!} \Big( \frac{c_1c_2}{n^2} \Big)^{k+\ell} \Big(1 + O(\ell^2/n) \Big) + O \Big( \! \sum_{\ell > \sqrt{n}} \frac{2^{\ell}}{\ell!} \Big) \\
&= \sum_{\ell \geq 0} \frac{(-1)^{\ell} 2^{\ell}}{\ell!} \Big( \frac{c_1c_2}{n^2} \Big)^{k+\ell} + O \bigg( n^{-1} \sum_{\ell \geq 0} \frac{\ell^2 2^{\ell}}{\ell!} \Big( \frac{c_1c_2}{n^2} \Big)^{k+\ell} \bigg) + O \Big( \frac{2^{\lceil\sqrt{n}\rceil}}{\lceil\sqrt{n}\,\rceil !} \Big) \\
&= \Big( \frac{c_1c_2}{n^2} \Big)^k \exp \Big( - \frac{2c_1c_2}{n^2} \Big) + O(n^{-1}).
\end{align*}
If either $B_1$ or $B_2$ has fewer than $n/8$ cherries (which happens with probability $O(n^{-1})$ if $B_1,B_2$ are randomly selected), then we can trivially estimate the number of matchings with exactly $k$ matched cherries by $n!$. It follows that the number of triples of two plane binary trees and a matching between their leaves such that there are exactly $k$ matched cherries is
$$C_n^2 \cdot n! \cdot \frac{2^k}{k!} \bigg( \E_n^{(P)} \bigg( \Big( \frac{c(B_1)c(B_2)}{n^2} \Big)^k \exp \Big( - \frac{2c(B_1)c(B_2)}{n^2} \Big) \bigg) + O(n^{-1}) \bigg).$$
Taylor expansion gives us
$$(x_1x_2)^k e^{-2x_1x_2} = 16^{-k} e^{-1/8} \bigg( 1 + \frac{8k-1}{2} \Big( \Big(x_1-{\displaystyle \frac14}\Big) + \Big(x_2 - \frac14 \Big) \Big) + O \Big( \Big(x_1 - \frac14 \Big)^2 + \Big(x_2 - \frac14 \Big)^2 \Big) \bigg).$$
Recall that under $\nup_n$, $c(B_1)$ and $c(B_2)$ are independent with mean $n(n-1)/(4n-6) = n/4 + O(1)$ and variance $O(n)$, so this yields
$$\E_n^{(P)} \bigg( \Big( \frac{c(B_1)c(B_2)}{n^2} \Big)^k \exp \Big( - \frac{2c(B_1)c(B_2)}{n^2} \Big) \bigg) = 16^{-k} e^{-1/8} + O(n^{-1}).$$
Finally we find that there are
\begin{equation}\label{eq:triples}
C_n^2 \cdot n! \cdot \frac{e^{-1/8}}{8^k k!} \big( 1 + O(n^{-1}) \big)
\end{equation}
triples  of two plane binary trees and a matching between their leaves with exactly $k$ matched cherries. For each of them, the size of the automorphism group is at least $2^k$, so the associated weight is at least $2^k/2^{2(n-1)}$. Consequently, the probability that there are exactly $k$ matched cherries is at least
$$\frac{1}{t_n} \cdot C_n^2 \cdot n! \cdot \frac{e^{-1/8}}{8^k k!} \big( 1 + O(n^{-1}) \big) \cdot \frac{2^k}{2^{2(n-1)}} = \frac{e^{-1/4} 4^{-k}}{k!} \big( 1 + O(n^{-1}) \big)$$
in view of~\eqref{eq:asymp}. Letting $p_{n,k}$ denote the probability that a random tanglegram of size $n$ has exactly $k$ matched cherries, we obtain
$$\liminf_{n \to \infty} p_{n,k} \geq  \frac{e^{-1/4} 4^{-k}}{k!}.$$
Since the sum of these lower bounds is already $1$, $\limsup_{n \to \infty} p_{n,k}$ cannot be any greater, so we must have
\begin{equation}\label{eq:limit_prob}
\lim_{n \to \infty} p_{n,k} = \frac{e^{-1/4} 4^{-k}}{k!},
\end{equation}
completing the proof of the first statement. This can now be refined somewhat further: let $\epsilon_{n,k}$ be the proportion of those triples of two plane binary trees and a matching counted by~\eqref{eq:triples} that give rise to at least one automorphism not induced by matched cherries (so that the total number of automorphisms is at least $2^{k+1}$). We find that the total weight of all triples with exactly $k$ matched cherries is at least
\begin{multline*}
C_n^2 \cdot n! \cdot \frac{e^{-1/8}}{8^k k!} \big( 1 + O(n^{-1}) \big) \cdot \Big( (1-\epsilon_{n,k}) \frac{2^k}{2^{2(n-1)}} + \epsilon_{n,k} \frac{2^{k+1}}{2^{2(n-1)}} \Big) \\
= t_n \cdot \frac{e^{-1/4} 4^{-k}}{k!} \big( 1 + O(n^{-1}) \big) \big( 1 + \epsilon_{n,k} \big).
\end{multline*}
In view of~\eqref{eq:limit_prob}, this means that $\limsup_{n \to \infty} \epsilon_{n,k} = 0$. Thus if $\tilde{p}_{n,k}$ is the probability that there are exactly $k$ matched cherries in a random tanglegram of size $n$ and all automorphisms are induced by them, we still have
$$\liminf_{n \to \infty} \tilde{p}_{n,k} \geq  \frac{e^{-1/4} 4^{-k}}{k!},$$
and the claim follows.
\end{proof}

\begin{remark}
Once again, we observe that only those automorphisms generated by cherries are asymptotically relevant (cf.~Lemma~\ref{lem:auto_bounds}).
If every tanglegram had trivial automorphism group, the number of tanglegrams of size $n$ would simply be
$$n! \cdot \Big( \frac{C_n}{2^{n-1}} \Big)^2.$$
Nontrivial automorphisms are thus only responsible for the factor $e^{1/8}$ in the asymptotic formula~\eqref{eq:asymp}, and as the proof of Theorem~\ref{thm:matched_cherries} shows, these nontrivial automorphisms are mostly generated by matched cherries.

The heuristic interpretation can be taken further: as explained earlier, two randomly generated plane binary trees with $n$ leaves will each have about $n/4$ cherries (with high probability). Suppose that the matching between the leaves is unbiased (uniformly random). Then a cherry in the left tree has a probability of approximately 
$$\frac{2 \cdot n/4}{n^2} = \frac{1}{2n}$$
of being matched to a cherry in the right tree. The law of rare events suggests that the distribution of matched cherries should converge to a Poisson distribution with mean
$$\frac{n}{4} \cdot \frac{1}{2n} = \frac18,$$
so the limiting probability that there are exactly $k$ matched cherries should be $\frac{e^{-1/8} 8^{-k}}{k!}$. However, $k$ matched cherries raise the weight by (at least, with high probability exactly) $2^k$, giving a total expected weight of
$$\sum_{k \geq 0} \frac{e^{-1/8} 8^{-k}}{k!} \cdot 2^k = e^{1/8}.$$
This is exactly the factor $e^{1/8}$ in the asymptotic formula for $t_n$. Moreover, because of this induced bias, the limiting probability for $k$ matched cherries is proportional to $\frac{e^{-1/8} 8^{-k}}{k!} \cdot 2^k = \frac{e^{-1/8} 4^{-k}}{k!}$, giving us, as Theorem~\ref{thm:matched_cherries} indeed shows, a Poisson distribution with mean $\frac14$.
\end{remark}

\bibliographystyle{abbrv}
\bibliography{Tanglegrams}

\end{document}